\documentclass[11pt]{amsart}
\usepackage{a4wide}
\usepackage{amssymb}
\usepackage{mathrsfs}
\usepackage{a4wide}
\usepackage{amssymb}
\usepackage{mathrsfs}
\usepackage{enumerate}
\theoremstyle{plain}
\theoremstyle{plain}
\newtheorem{theorem}{Theorem}[section]
\newtheorem*{theorem*}{Theorem}%[section]
\newtheorem{lemma}[theorem]{Lemma}
\newtheorem{proposition}[theorem]{Proposition}
\theoremstyle{definition}
\newtheorem{definition}[theorem]{Definition}%[section]
\newtheorem{remark}[theorem]{Remark}%[section]
\newtheorem{question}[theorem]{Question}%[section]
\usepackage{lmodern} %pour de plus belles polices
  \newcommand{\C}{\ensuremath{\mathbb{C}}}%
  
  \newcommand{\R}{\ensuremath{\mathbb{R}}}%
  \newcommand{\N}{\ensuremath{\mathbb{N}}}%
  \newcommand{\Z}{\ensuremath{\mathbb{Z}}}%
\begin{document}
\title[Invariant means for the wobbling group]{Invariant means for the wobbling group}
\author[]{Kate Juschenko, Mikael de la Salle}
\address{Kate Juschenko: Northwestern University \& CNRS, Universit\'e Lyon 1}
\address{Mikael de la Salle: CNRS, ENS de Lyon.
}
\thanks{The work of M. de la Salle was partially supported by ANR grants OSQPI and NEUMANN}
\begin{abstract}
Given a metric space $(X,d)$, the wobbling group of $X$ is the group of bijections $g:X\rightarrow X$ satisfying $\sup\limits_{x\in X} d(g(x),x)<\infty$. We study algebraic and analytic properties of $W(X)$ in relation with the metric space structure of $X$, such as amenability of the action of the lamplighter group $ \bigoplus_{X} \Z/2\Z \rtimes W(X)$ on $\bigoplus_{X} \Z/2\Z$ and property~(T).
\end{abstract}
\maketitle

\section{Introduction}
In this paper we deal with \emph{amenable actions} of discrete groups. In our setting  an action of a group $G$ on a set $X$ is called \emph{amenable} if there is an $G$-invariant mean on $X$. A linear map $\mu$ on $\ell_\infty(X)$ is a \emph{mean} on $X$ if it is unital and $\|\mu\|=1$. A group $G$ is amenable if and only if its action on itself by left translation is amenable, in this case all actions of $G$ are amenable. Thus the question of determining whether an action is amenable is interesting in the case when $G$ is not (known to be) amenable.

Let $G$ be a discrete group acting transitively on a set $X$. The abelian group $\bigoplus_{X} \Z/2\Z$ carries an action of itself by translation, and an action of $G$ by permutation of the basis, which gives rise to an action of the semidirect product (also called permutational wreath product, or lamplighter group) $\bigoplus_{X} \Z/2\Z \rtimes G$. We will be interested in particular cases of the following general question~:

\begin{question}\label{MQ1} {\it Is the action of $\bigoplus_X \Z/2\Z  \rtimes G$ on $\bigoplus_X \Z/2\Z$ amenable?}
\end{question}

An easy necessary condition for \ref{MQ1} is that the action of $G$ on $X$ is amenable. We observe that this is not sufficient, see Proposition \ref{NCQ1}.

In \cite{JNS}, Nekrashevych and the authors showed that Question \ref{MQ1} has a positive answer if the Schreier graph of the action of $G$ on $X$ is recurrent. However a characterization of the actions for which the answer is positive is still open.

In this note all metric spaces will be discrete. A metric space $X$ has bounded geometry if for every $R>0$, the balls of radius $R$ have bounded cardinality. We will mainly be interested in a special case of Question \ref{MQ1} when $(X,d)$ is a metric space with bounded geometry and $G$ is a group of bijections $g$ of $X$ with bounded displacement, \emph{i.e.} with the property that $|g|_w<\infty$, where
\begin{equation}\label{eq=wobbling} |g|_w:=\sup \{d(x,g(x)):x\in X\}.\end{equation} Following \cite{deuber} (see also \cite{CGH}) we will call the group of all such bijections of $X$ the \emph{wobbling group} of $X$ and denote it by $W(X)$. In \cite{lacz}, \cite[Remark $0.5.C_1''$]{gromov} and \cite{deuber} the wobblings were introduced as tools to prove non-amenability results. In \cite{JM}, they were used to prove amenability results (see below for details). When $X$ is a Cayley graph of a finitely generated group $\Gamma$ with word metric we will denote the wobbling group of $X$ shortly by $W(\Gamma)$. The group $W(\Gamma)$ does not depend on a finite generating set of $\Gamma$ and it coincides with the group of piecewise translations of $\Gamma$. As a special case of Question \ref{MQ1} we can ask~:

\begin{question}\label{main_question} {\it Is the action of $\bigoplus_X \Z/2\Z \rtimes W(X)$ on $\bigoplus_X \Z/2\Z$ amenable?}
\end{question}

The motivation for the question above is based on the recent result of the first named author and N. Monod, \cite{JM}, where the authors show that {\it the full topological group of Cantor minimal system} is amenable, which was previously conjectured by Grigorchuk and Medynets in \cite{GM3}. The combination of this result with the result of H. Matui, \cite{Mat}, produces the first examples of infinite simple finitely generated amenable groups. The technical core of \cite{JM} is to show that the Question \ref{main_question} has a positive answer for the particular case $X=\mathbb{Z}$.

Our goal would be to give a necessary and sufficient condition on $(X,d)$ for Question \ref{main_question} to have a positive answer. Theorem \ref{thm=summary_of_paper} summarizes our partial results in this direction.
\begin{definition}\label{def_transience} Let $(X,d)$ be a metric space with bounded geometry and fix $x_0 \in X$. $(X,d)$ is called transient if there is $R>0$ such that the random walk starting at $x_0$ and jumping from a point $x$ uniformly to $B(x,R)$ is transient. Otherwise it is called recurrent.
\end{definition}
This notion does not depend on $x_0$, and when $(X,d)$ is a connected graph with graph distance this notion is equivalent to the transience of the usual random walk on this graph (Proposition \ref{properties_transience}).
\begin{theorem}\label{thm=summary_of_paper} Let $(X,d)$ be a metric space with bounded geometry.
\begin{itemize}
\item If $(X,d)$ is recurrent, then the action of $\bigoplus_X \Z/2\Z \rtimes W(X)$ on $\bigoplus_X \Z/2\Z$ is amenable. This includes $X=\Z,\Z^2$ or more generally a metric space $(X,d)$ with bounded geometry that embeds coarsely in $\Z^2$.
\item If $X$ contains a Lipschitz and injective image of the infinite binary tree, then the action of $\bigoplus_X \Z/2\Z \rtimes W(X)$ on $\bigoplus_X \Z/2\Z$ is not amenable.
\end{itemize}
\end{theorem}
The sufficient condition in terms of the random walk uses \cite{JNS} and is a necessary and sufficient condition for a stronger condition to hold, see Remark \ref{rem=CNS}. We also take the opportunity in Remark \ref{rem=taka} to present an alternative proof, due to Narutaka Ozawa, of \cite[Theorem 1.2]{JNS}. 

By Remark \ref{groups_coarsely_containing_tree}, Question \ref{main_question} has a negative answer for many Cayley graphs of groups with exponential growth. By  \cite[Theorem 3.24]{Woess}, the first criterion applies to a finitely generated group $X=\Gamma$ if and only if $\Gamma$ is virtually $\{0\}$, $\Z$, $\Z^2$. The case when $X=\Z^d$, $d\geq 3$ remains an intriguing open question.

By \cite{Ozawa} a positive answer to Question \ref{main_question} would follow from the weak amenability of $\bigoplus_X \Z/2\Z \rtimes W(X)$. We could not follow this approach, but this led us to wonder whether $W(X)$ can contain property (T) subgroups.

As one may expect there is a strong relation between group structure of $W(X)$ and metric space structure of $X$. We show that if $X$ is of uniform subexponential growth, then $W(X)$ does not contain infinite property $(T)$ subgroups, see Theorem \ref{thm=WX_contains_no_T_subgroup}. On the other hand, an example of R. Tessera, see Theorem \ref{Rom} shows that there exists a solvable group $\Gamma$ such that $W(\Gamma)$ contains $SL_3(\mathbb{Z})$.

The paper is organized as follows. In Section \ref{RW} we study the notion of transience for metric spaces with bounded geometry and prove the first half of Theorem \ref{thm=summary_of_paper}. In Section \ref{section=neg} we prove the second half of Theorem \ref{thm=summary_of_paper} (Proposition \ref{X_contains_tree}), and in a last section we study when $W(X)$ contains property $(T)$ groups.
\section*{Acknowledgements} We thank R.~Tessera for interesting conversations and for allowing us to include Theorem \ref{Rom}. We also thank Y.~de Cornulier for useful remarks on the presentation of Section \ref{section_on_wobbling_group}. We are grateful to R.~Grigorchuk for providing us with numerous comments and suggestions that improved the paper a lot. In a preliminary version of this paper our results were not stated in terms of random walks. U.~Bader, B.~Hua and A.~Valette suggested a connection with recurrence of random walks and B. Hua pointed out \cite[Theorem 3.24]{Woess} to us. We thank them, and especially A.~Valette, for this suggestion and for many other valuable comments. Finally we thank Narutaka Ozawa for his Remark \ref{rem=taka}.

\section{Recurrent random walks and amenability}\label{RW}

Here we prove the following fact on recurrent random walks.
\begin{proposition}\label{prop=link_recurrence} Let $(X,d)$ be a metric space with bounded geometry and $x_0 \in X$. Then $(X,d)$ is recurrent if and only if for every finitely supported symmetric probability measure $\mu$ on $W(X)$, the random walk on $X$ starting from $x_0$ and jumping from $x$ to $g\cdot x$ according to the measure $\mu$ is recurrent.
\end{proposition}

Before we prove Proposition \ref{prop=link_recurrence}, we state some properties of transience for metric spaces.
\begin{proposition}\label{properties_transience} 
Let $(X,d)$ be a metric space with bounded geometry and $x_0 \in X$. Let $R>0$ such that the random walk starting at $x_0$ and jumping from a point $x$ uniformly to $B(x,R)$ is transient. Then for every $R'>R$ the random walk starting at $x_0$ and jumping from a point $x$ uniformly to $B(x,R)$ is also transient.

The notion of transience given by Definition \ref{def_transience} is independent of $x_0$.

In the case $(X,d)$ is a connected graph with bounded geometry, the transience in the sense of Definition \ref{def_transience} is equivalent to the transience of the usual random walk on the graph.
\end{proposition}
\begin{proof}
Consider $(V,E)$ the connected component of $x_0$ in the graph structure on $X$ where there is an edge between two points of $X$ at distance at most $R'$. Then the random walk starting at $x_0$ and jumping from $x$ uniformly to $B(x,R)$ is a reversible random walk on $(V,E)$ with constant conductance and bounded range, so that by \cite[Theorem 3.2]{Woess} its transience implies the transience of the simple random walk on $(V,E)$. This proves the first point.

Let $x_0,x_1 \in X$. If there is $R$ such that the random walk starting at $x_0$ and jumping from a point $x$ uniformly to $B(x,R)$ is transient, by the first point we can assume that $R>d(x,x_0)$, so that the same random walk starting at $x_1$ is also transient. This proves the second point.

Assume that $(X,E)$ is a connected graph with bounded geometry, take $x_0 \in X$ and $R\geq 1$. Let $(X,E')$ be the graph structure on $X$ in which there is an edge between two points of $X$ at distance at most $R$. The formal identity between $(X,d)$ and $(X,d')$ is a bilipschitz bijection, so that by \cite[Theorem 3.10]{Woess} the random walk on $(X,E)$ is transient if and only if the random walk on $(X,E')$ is transient.
\end{proof}

\begin{proof}[Proof of Proposition \ref{prop=link_recurrence}]
  Assume that $(X,d)$ is recurrent. Take $\mu$ as in the Proposition. Remember the notation \eqref{eq=wobbling} and pick $R>\max_{g \in \textrm{supp}(\mu)} |g|_w$. Since $(X,d)$ is recurrent, the random walk starting at $x_0$ and jumping from a point $x$ uniformly to $B(x,R)$ is recurrent. By \cite[Theorem 3.2]{Woess} the random walk starting at $x_0$ and jumping from a point $x$ to $g\cdot x$ uniformly according to $\mu$ is therefore also recurrent.

Reciprocally, assume that $(X,d)$ is transient, and take $R>0$ as in the definition. We will construct a finite symmetric subset $S$ of $W(X)$ such for that every pair of points $x,y \in X$ at distance less than $R$ there is $g \in S$ such that $g x = y$. By \cite[Theorem 3.2]{Woess} this will imply the transience of the simple random walk on the connected component of $x_0$ in the graph structure on $X$ in which there is an edge between $x$ and $gx$ for every $x \in X, g \in S$. In other words if $\mu$ is the uniform probability measure on $S$, the random walk on $X$ starting from $x_0$ and jumping from $x$ to $g\cdot x$ according to $\mu$ is transient. Here is the construction of $S$. Define a graph structure on $X$ by putting
an edge between $x$ and $x'$ if $d(x,x')\leq R$. We obtain a (not
necessarily connected) graph $(X,E)$ with bounded geometry on which the random walk starting from $x_0$ is transient. Denote by $d_E$ the
associated graph distance. Take a finite collection $(X_i)_{i \leq l}$ of subsets of $X$
such that $\cup_i X_i = X$ and $d_E(x,y) \geq 3$ for all $x,y \in X_i$
and all $i$. Take $k \in \N$, and for every $x \in X$ take a sequence
$y_1(x),\dots,y_k(x)$ that covers all neighbours of $x$ in
$(X,E)$. The existence of such collection $(X_i)$ and such $k$ follows
from the bounded geometry assumption. Then for every $i\leq l$ and
every $j \leq k$, consider the element $s_{i,j}$ of $W(X)$ that
permutes $x$ and $y_j(x)$ for every $x \in X_i$ and acts as the
identity on the rest of $X$. Then $S= \{s_{i,j},i \leq l, j\leq k\}$
works. Indeed by construction for every neighbours $(x,x') \in (X,E)$
there is at least one (in fact two) element of $S$ that permutes $x$ and $x'$.
\end{proof}

\begin{proof}[Proof of the first part of Theorem \ref{thm=summary_of_paper}] Assume that $(X,d)$ is recurrent. Let $G$ be a finitely generated subgroup of $W(X)$. By Proposition \ref{prop=link_recurrence} and \cite[Theorem 1.2]{JNS}, the action of $\bigoplus_{G x_0} \Z/2\Z \rtimes G$ on $\bigoplus_{G x_0} \Z/2\Z$ is amenable. This implies that the action of $\bigoplus_{X} \Z/2\Z \rtimes G$ on $\bigoplus_{X} \Z/2\Z$ is amenable~: if $m$ is a $\bigoplus_{G x_0} \Z/2\Z \rtimes G$-invariant mean on $\ell_\infty(\bigoplus_{G x_0} \Z/2\Z)$, then $f \in \ell_\infty(\bigoplus_{X} \Z/2\Z) \mapsto m(f\left|_{\bigoplus_{G x_0} \Z/2\Z}\right.)$ is a $\bigoplus_{X} \Z/2\Z \rtimes G$-invariant mean. This proves that the action of every finitely subgroup of $\bigoplus_{X} \Z/2\Z \rtimes W(X)$ on $\bigoplus_{X} \Z/2\Z$ is amenable, and concludes the proof. 
\end{proof}
\begin{remark}\label{rem=CNS} As is well-known, the action of a group on a set $Y$ is amenable if and only if there is a net $f_\alpha$ of unit vectors in $\ell_2(Y)$ such that $\lim_\alpha \| g \cdot f_\alpha -f_\alpha \|=1$ for all $g \in G$. In the special case of $G = \bigoplus_X \Z/2\Z \rtimes W(X)$ acting on $Y =\bigoplus_X \Z/2\Z$, Proposition \ref{prop=link_recurrence} and \cite[Theorem 1.2]{JNS} show that the recurrence of $(X,d)$ is equivalent to the existence of such a net $f_\alpha$ with the additional property that $f_\alpha \in \ell^2(\bigoplus_X \Z/2\Z)$ is of the form $f_\alpha(\omega) = \prod_{x \in X} f_{\alpha,x}(\omega_x)$ for functions $f_{\alpha,x} \colon \Z/2\Z \to \C$.
\end{remark}
\begin{remark}\label{rem=taka} By the theory of electrical networks (see \cite{JNS} for details), the Schreier graphs of a transitive action of a finitely generated group $G$ on $X$ carry a recurrent random walk if and only there exists a sequence of finitely supported function $a_n:X\rightarrow [0,1]$ that satisfy~:
\begin{enumerate}
\item $a_n(x_0)=1$,
\item $\lim_n \|g\cdot a_n-a_n\|_{\ell_2(X)}<\varepsilon$ for every $g\in G$.
\end{enumerate}
It was proved in \cite{JNS} (and used above) that this implies a positive answer to Question \ref{MQ1}. We record here a slightly different proof, due to Narutaka Ozawa (personal communication). Let $\mathcal P_f(X)$ denote the set of all finite subsets of $X$, that we identify with $\oplus_X \Z/2\Z$. For $a_n$ as above, let $\xi_n(B)=\prod\limits_{x\in B} a_n(x)$ for $B\in \mathcal{P}_f(X)$ and $\xi_n(\emptyset)=1$. Then $\xi_n \in \ell^2(\mathcal P_f(X))$ is $\{x_0\}$-invariant and for $g \in G$
\[ \langle g\xi_n,\xi_n \rangle = \sum_{B \in \mathcal P_f(X)} \prod_{x \in B} a_n(x) a_n(gx) = \prod_{x \in X} (1+a_n(x)a_n(gx))\]
by distributivity. In particular for the identity element,
\[ \langle \xi_n,\xi_n \rangle  = \prod_{x \in X} (1+a_n(x)^2) = \prod_{x \in X} (1+a_n(gx)^2) = \prod_{x \in X} \sqrt{1+a_n(x)^2} \sqrt{1+a_n(gx)^2}\]
by reordering the terms. Therefore
\begin{multline*}\log{\frac{\langle \xi_n,\xi_n \rangle}{\langle g\xi_n,\xi_n \rangle}}=\log{\prod\limits_{x\in X}} \frac{\sqrt{1+a_n(x)^2} \sqrt{1+a_n(gx)^2}}{1+a_n(x)a_n(gx)}\\ \leq \sum_{x \in X} \frac{(a_n(x) - a_n(gx))^2}{2(1+a_n(x)a_n(gx))^2}\leq \frac{1}{2}\|a_n-g\cdot a_n\|_{\ell^2(X)}^2,\end{multline*}
which goes to zero as $n \to \infty$. The first inequality is the basic inequality $\ln(\sqrt{1+A}) \leq \frac 1 2 A$ for 
\[ A = \frac{(1+a_n(x)^2)(1+a_n(gx)^2)}{(1+a_n(x)a_n(gx))^2} -1 = \frac{(a_n(x) - a_n(gx))^2}{(1+a_n(x)a_n(gx))^2}.\]
 Any weak-$*$ cluster point in $\ell_\infty(\mathcal P_f(X))^*$ of the sequence $|\xi_n|^2/\|\xi_n\|^2$ will therefore be a $G\ltimes \mathcal{P}_f(X)$-invariant mean. This construction of $\xi_n$ should be compared to the one in \cite{JNS}, which was defined (through Fourier transform) as $\xi_n(B) = \prod_{x \in B} \sin(\frac{\pi}{4} a_n(x)) \times \prod_{x \notin B} \cos(\frac{\pi}{4} a_n(x))$.
\end{remark}

\section{Negative answer to the Question \ref{main_question}}\label{section=neg}
\label{section_on_wobbling_group}

Let $G$ be a group acting on $X$. We start by recording the following result. The second assertion follows from results proved later, but is not used in the rest of the paper.
\begin{proposition}\label{NCQ1} If the action of $\bigoplus_X \Z/2\Z \rtimes G$ on $\bigoplus_X \Z/2\Z$ is amenable, then so is the action of $G$ on $X$. The converse is not true.
\end{proposition}
\begin{proof}
Assume that the action of $ \bigoplus_X \Z/2\Z \rtimes G$ on $\bigoplus_X \Z/2\Z$ amenable. By (ii) implies (iv) in \cite[Lemma 3.1]{JM}, the set $\mathcal P_f^*$ of non-empty finite subsets of $X$ carries a $G$-invariant mean $m$. Consider the unital positive $G$-equivariant map $T:\ell_\infty(X) \to \ell_\infty(\mathcal P_f^*)$ given by $Tf(A)$ is the average of $f$ on $A$, for all $A$ nonempty finite subset of $X$.  The composition $m \circ T$ is a $G$-invariant mean on $X$.

To see that the converse is not true, take for $X$ the Cayley graph of a finitely generated amenable group $\Gamma$ that contains an infinite binary tree (see Remark \ref{groups_coarsely_containing_tree} for the existence of such group). By Theorem \ref{thm=summary_of_paper}, the action of $ \bigoplus_X \Z/2\Z \rtimes W(X)$ on $\bigoplus_X \Z/2\Z$ is not amenable. On the other hand the action of $W(X)$ on $X$ is amenable; more precisely any $\Gamma$-invariant mean $m$ on $X$ is also $W(X)$-invariant. Indeed, for any $g \in W(X)$ there is a finite partition $A_1,\dots,A_n$ of $X$ and elements $\gamma_1,\dots,\gamma_n$ such that $g$ acts as the translation by $\gamma_k$ on $A_k$. Then for every $f \in \ell_\infty(X)$, using that $(\gamma_i(A_i))_{i=1}^n$ forms a partition of $X$ we get
\[m(g \cdot f) = \sum_i m(\gamma_i \cdot (f 1_{\gamma_i(A_i)}) = \sum_i m(f 1_{A_i}) = m(f).\]
\end{proof}

When $X$ is the Cayley graph of a finitely generated group, the first assertion in Proposition \ref{NCQ1} implies
\begin{lemma}\label{no_mean_when_nonamenable}
Let $\Gamma$ be a finitely generated group. If there exists a $\bigoplus_\Gamma \Z/2\Z \rtimes W(\Gamma)$-invariant mean on $\bigoplus_\Gamma \Z/2\Z$ then $\Gamma$ is amenable.
\end{lemma}
We can also give a negative answer to Question \ref{main_question} for some amenable groups. One ingredient for this is the following monotonicity property.
\begin{lemma} \label{monotonicity} Let $i:X \to Y$ an injective map such that $\sup_{d(x,x')\leq R} d(i(x),i(x'))<\infty$ for every $R>0$. If Question
\ref{main_question} has a positive answer for $Y$, then is also has
positive answer for $X$.
\end{lemma}
\begin{proof}
In this proof we denote by $\mathcal P_f(X)$ the set of all finite subsets of $X$, which carries a natural action of $W(X)$. It follows from the equivalence of (ii) and (iv) in \cite[Lemma 3.1]{JM} that Question \ref{main_question} has a positive answer if and only if there is a $W(X)$-invariant mean on $\mathcal P_f(X)$ giving full weight to the subsets containing any given element of $X$.

The map $i$ allows to define an embedding $W(X) \subset W(Y)$ by
defining, for $g \in W(X)$, $g\cdot i(x) = i(g\cdot x)$ and $g \cdot y
= y$ if $y \notin i(X)$.

Assume that Question \ref{main_question} has a positive answer for $Y$, and take $x_0 \in X$. By \cite[Lemma 3.1]{JM} there is a mean $m$ on $\mathcal P_f(Y)$ that is $W(Y)$-invariant and that gives full weight to the collection
of sets containing $i(x_0)$. Then the push-forward mean on $\mathcal P_f(X)$ (given by $\varphi \in \ell_\infty(\mathcal P_f(X)) \mapsto m(A \mapsto \varphi(i^{-1}(A))$) is $W(X)$-invariant and gives full weight to the collection
of sets containing $x_0$. By \cite[Lemma 3.1]{JM} again, Question \ref{main_question} has a positive answer for $X$.
\end{proof}

Lemma \ref{monotonicity} and Proposition \ref{NCQ1} imply that for $\bigoplus_X \Z/2\Z$ to act amenably on $\bigoplus_X \Z/2\Z$ it is necessary that $W(X')$ act amenably on $X'$ for all $X' \subset X$. In particular the following Proposition establishes the second half of Theorem \ref{thm=summary_of_paper}.
\begin{proposition} \label{X_contains_tree}
Let $(X,d)$ be a metric space with bounded geometry with an injective and Lipschitz map from the infinite binary tree $T$ to $X$. Then there is no $\bigoplus_X \Z/2\Z \rtimes W(X)$-invariant mean on $\bigoplus_X \Z/2\Z$.
\end{proposition}
\begin{proof}
There is a Lipschitz injective map from the free group with two generators in $T$, and hence in $X$ if $X$ contains an injective and Lipschitz image of $T$. The Proposition therefore follows from Lemma
\ref{no_mean_when_nonamenable} and Lemma \ref{monotonicity}.
\end{proof}
\begin{remark}\label{groups_coarsely_containing_tree} The class of groups for which this proposition applies, \emph{i.e.} for which there is a Cayley graph that contains a copy of the infinite binary tree as a subgraph, contains in particular all non-amenable groups (\cite[Theorem 1.5]{Benjamini-Schramm}), as well as all elementary amenable groups with exponential growth (by \cite{C} such groups contain a free subsemigroup). In \cite{Grig}, R. Grigorchuk, disproving a conjecture of Rosenblatt, proved that the lamplighter group $Z_2 \wr G$ contains an infinite binary tree, here $G$ is Grigorchuk's $2$-group of intermediate growth. We do not know whether all groups with exponential growth contain such a tree.
\end{remark}

\section{Property (T) subgroups}
It is an interesting question to extract properties of the group $W(X)$ using the properties of the underlying metric space. Below we prove that $W(X)$ cannot contain property $(T)$ groups when $X$ is of subexponential growth. Alain Valette (personal communication) pointed out to us that a very similar observation (atttributed to Kazhdan) was made by Gromov in \cite{gromov} Remark 0.5.F: a discrete property (T) group $G$ cannot contain a subgroup $G'$ such that $G/G'$ has subexponential growth unless $G/G'$ is finite.

\begin{theorem}\label{thm=WX_contains_no_T_subgroup}
Let $X$ be a metric space with uniform subexponential growth~:
\[\lim_{n} \frac 1 n \log \sup_{x \in X} |B(x,n)| =0.\] Then $W(X)$ does not contain
an infinite countable property $(T)$ group.
\end{theorem}
\begin{proof}
Assume $G<W(X)$ is a finitely generated property $(T)$ group, with finite symmetric generating set $S$. We will prove that $G$ is finite. To do so we prove that the $G$-orbits on $X$ are finite, with a uniform bound. Assume that $1 \in S$. If $m=max\{|g|_w:g\in S\}$, then $S^n x \subset B(x,mn)$ for every $x \in X$, so that by assumption, the growth of $S^n x$ is subexponential (uniformly in $x \in X$). The classical expanding properties for actions of $(T)$ groups will imply that the orbit of $x$ is finite (uniformly in $x$).

Indeed, by $(T)$, there exists $\varepsilon>0$ such that for every unitary action of $G$ on a Hilbert space $H$ without invariant vectors, the inequality $\sum_{g \in S} \|g \cdot \xi - \xi\|^2 \geq \varepsilon \|\xi\|^2$ holds for every $\xi \in H$. As a consequence, for every transitive action of $G$ on a set $Y$, we have $\sum_{g \in S} |g F \Delta F| \geq \varepsilon/2 |F|$ for every finite subset $F$ of $Y$ satisfying $2 |F| \leq |Y|$ (take $H=\ell_2(Y)$ if $Y$ is infinite, and $H=$the subspace of $\ell_2(Y)$ orthogonal to the vector with all coordinates equal otherwise, and apply the preceding equality with $\xi = \chi_F - |F|/|Y \setminus F| \chi_{Y \setminus F}$. Here $\chi_F$ is the indicator function of $F$, and $|F|/|Y \setminus F|$ is by convention $0$ if $Y$ is infinite). By induction, we therefore have that for $x\in Y$ and $n \in \N$, $|S^n x| \geq (1+\varepsilon/4)^n$ unless $|Y| \leq 2(1+\varepsilon/4)^n$. Applying it to the orbit of some $x \in X$, we get
\[
|S^n x| < (1+\varepsilon/4)^n \Longrightarrow |Orb_G(x)| < 2(1+\varepsilon/4)^n.
\]
Hence, subexponential growth gives an $n \in \N$ such that $|Orb_G(x)| < 2(1+\varepsilon/4)^n$ for every $x \in X$. QED.
\end{proof}

To construct spaces such that $W(X)$ contains property (T) groups, we first remark that the groups $W(X)$ behave well with respect to coarse embeddings. A map $q:(X,d_X) \to (Y,d_Y)$ between metric spaces is {\it a coarse embedding} if there exists nondecreasing functions $\varphi_+,\varphi_-:[0,\infty[ \to \R$ such that $\lim_{t \to \infty} \varphi_-(t) = \infty$ and
$$\varphi_-(d_X(x,x')) \leq d_Y(q(x),q(x')) \leq \varphi_+(d_X(x,x'))$$ 
for every $x,x' \in X$.
\begin{lemma}\label{lemma=WX_contained_in_WY}
Let $q:(X,d_X) \to (Y,d_Y)$ be a map such that there is an increasing function
$\varphi_+:\R^+ \to \R^+$ such that $d_Y(q x, q y) \leq
\varphi_+(d_X(x,y))$, and such that the preimage $q^{-1}(y)$ of every $y
\in Y$ has cardinality less than some constant $K$ (e.g. $q$ is a
coarse embedding and $X$ has bounded geometry). Let $F$ be a finite
metric space of cardinality $K$. Then $W(X)$ is isomorphic to a
subgroup of $W(Y \times F)$.
\end{lemma}
\begin{proof}
In this statement $Y \times F$ is equipped with the distance $d(
(y,f),(y',f')) = d_Y(y,y')+d_F(f,f')$. Since $F$ is bigger than
$q^{-1}(y)$ for all $y$, there is a map $f:X \to F$ such that the map
$\widetilde q: x \in X \mapsto (q(x),f(x))\in Y \times F$ is
injective. We can therefore define an action of $W(X)$ on $Y \times F$
by setting $g(\widetilde q(x))=\widetilde q(g x)$ and
$g (y,f)=(y,f)$ if $(y,f) \notin \widetilde q(X)$. The
assumption on $\varphi_+$ guarantees that this action is by wobblings,
ie that it defines an embedding of $W(X)$ in $W(Y \times F)$.
\end{proof}

In a contrast to Theorem \ref{thm=WX_contains_no_T_subgroup} we have the following result by Romain Tessera. With his kind permission we include a proof.
\begin{theorem} \label{Rom}
There is a solvable group $\Gamma$ such that $W(\Gamma)$ contains the
property $(T)$ group $SL(3,\Z)$.
\end{theorem}
\begin{proof}
The proof uses the notion of asymptotic dimension (see \cite{BD}). By
\cite[Corollary 94]{BD}, $SL(3,\Z)$ has finite asymptotic
dimension. By \cite[Theorem 44]{BD} this implies that $SL(3,\Z)$
embeds coarsely into a finite product of binary trees. Take $\Gamma_0$ a
solvable group with a free semigroup. In particular it coarsely
contains a binary tree, so $SL(3,\Z)$ embeds coarsely in $\Gamma_0^n$
for some $n$. By Lemma \ref{lemma=WX_contained_in_WY}, there is a
finite group $F$ such that $W(SL(3,\Z))$ embeds as a subgroup in $W(F
\times\Gamma_0^n)$. But $W(SL(3,\Z))$ contains $SL(3,\Z)$ (action by
translation).
\end{proof}
\begin{remark} The proof actually shows that for every group $\Lambda$ with
finite asymtotic dimension, there is an integer $n$ such that
$\Lambda$ is isomorphic to a subgroup of $W(\Gamma^n)$ whenever there
is a Cayley graph of $\Gamma$ that contains an infinite binary tree as
a subgraph. By Remark \ref{groups_coarsely_containing_tree} this includes lots of groups $\Gamma$ with exponential growth. In some sense this says that the assumptions of Theorem \ref{thm=WX_contains_no_T_subgroup} are not so restrictive.
\end{remark}

\end{document}